\documentclass[11pt]{article}
\usepackage{latexsym}
\usepackage{amssymb}
\usepackage{amsmath}
\usepackage{amsthm}
\usepackage{bm}
\theoremstyle{plain}
\newtheorem{thm}{Theorem}[section]
\newtheorem{lem}[thm]{Lemma}

\newtheorem{rem}[thm]{Remark}
\newtheorem{defin}[thm]{Definition}

\newcommand{\RR}{{\mathbb R}}

\newcommand{\ep}{\varepsilon}

\newcommand{\tr}{{\rm tr}}

\numberwithin{equation}{section}
\begin{document}
\newcounter{aaa}
\newcounter{bbb}
\newcounter{ccc}
\newcounter{ddd}
\newcounter{eee}
\newcounter{ggg}
\newcounter{xxx}
\newcounter{xvi}
\newcounter{x}
\setcounter{aaa}{1}
\setcounter{bbb}{2}
\setcounter{ccc}{3}
\setcounter{ddd}{4}
\setcounter{ggg}{7}
\setcounter{eee}{32}
\setcounter{xxx}{10}
\setcounter{xvi}{16}
\setcounter{x}{38}
\title
{Wong-Zakai approximation of solutions to reflecting stochastic differential
equations on domains in Euclidean spaces II}
\author{Shigeki Aida\footnote{
This research was partially supported by
Grant-in-Aid for Scientific Research (B) No.~24340023.}
\\
Mathematical Institute\\
Tohoku University,
Sendai, 980-8578, JAPAN\\
e-mail: aida@math.tohoku.ac.jp}
\date{}
\maketitle
\begin{abstract}
The strong convergence of 
Wong-Zakai approximations of the solution
to the reflecting stochastic differential
equations was studied in \cite{aida-sasaki}.
We continue the study 
and prove the strong convergence under weaker assumptions
on the domain.
\end{abstract}

\section{Introduction}

Wong-Zakai approximations of solutions of stochastic differential
equations (=SDEs) were studied by many researchers, {\it e.g.}
\cite{ikeda-watanabe, wong-zakai, gyongy-stinga}.
In the case of reflecting SDEs, Doss and Priouret~\cite{doss}
studied the Wong-Zakai approximations when the boundary is smooth.
Actually, the unique existence of strong solutions of
reflecting SDEs were proved for domains whose boundary may not be smooth
by Tanaka~\cite{tanaka}, Lions-Sznitman~\cite{lions-sznitman}
and Saisho~\cite{saisho}.
In their studies, the standard conditions,
(A), (B), (C) and admissibility condition,
on the domain for reflecting SDEs
were introduced and the unique existence of strong solutions
were proved under the conditions either
(A) and (B) hold or the domain is convex in
\cite{tanaka} and \cite{saisho}.
We explain the conditions (A), (B), (C) in the next section.
There were studies on Wong-Zakai approximations in such cases,
{\it e.g.}, \cite{pettersson, ren-xu1, ren-xu2} for convex domains
and \cite{evans-stroock} for domains satisfying 
admissibility condition as well as conditions (A),
(B), (C).
When the domain is convex,
Ren and Xu~\cite{ren-xu2} proved that Wong-Zakai approximations
converge to the true solution in probability in the setting of
stochastic variational inequality.
In \cite{aida-sasaki}, the strong convergence of
Wong-Zakai approximations
was proved under the conditions (A), (B), (C).
We note that Zhang~\cite{zhang} proved the strong
convergence of Wong-Zakai approximations
in the setting of \cite{evans-stroock} independent
of \cite{aida-sasaki}.
The aim of this paper is to prove the strong convergence 
of Wong-Zakai approximation under the conditions either
(A) and (B) hold or the domain is convex
following the proof in \cite{aida-sasaki}.
Note that our proof in the case of convex domains is different from \cite{ren-xu2}
and we give an estimate of the order of convergence.

The paper is organized as follows.
In Section 2, we recall conditions of the boundary 
and state the main theorems.
The first main theorem (Theorem~\ref{convex case}) shows the strong convergence of
Wong-Zakai approximations when the domain is
convex.
The estimate of the order of the convergence is the same as
given in \cite{aida-sasaki}.
The second main theorem (Theorem~\ref{conditions A B}) is concerned with
the convergence of Wong-Zakai approximations in
the case where the domain satisfies the conditions (A) and (B).
We prove main theorems
in Section 3 and Section 4.

\section{Preliminaries and main theorems}

Let $D$ be a connected domain in $\RR^d$.
The following conditions can be found 
in \cite{saisho}.
In \cite{aida-sasaki}, we used the conditions
(A), (B), (C) on $D$.
In this paper, we will use (B') too.
The set ${\cal N}_x$ 
of inward unit normal vectors at 
$x\in \partial D$ is defined by
\begin{align*}
 {\cal N}_x&=\cup_{r>0}{\cal N}_{x,r},\\
{\cal N}_{x,r}&=\left\{{\bm n}\in \RR^d~|~|{\bm n}|=1,
 B(x-r{\bm n},r)\cap D=\emptyset\right\},
\end{align*}
where $B(z,r)=\{y\in \RR^d~|~|y-z|<r\}$, $z\in \RR^d$, $r>0$.

\begin{defin}
\begin{enumerate}
 \item[{\rm (A)}] $(\mbox{uniform exterior sphere condition})$.
 There exists a constant
$r_0>0$ such that
\begin{align}
{\cal N}_x={\cal N}_{x,r_0}\ne \emptyset \quad \mbox{for any}~x\in
 \partial D.
\end{align}

\item[{\rm (B)}]
There exist constants $\delta>0$ and $\beta\ge 1$
satisfying:

for any $x\in\partial D$ there exists a unit vector $l_x$ such that
\begin{align}
 (l_x,{\bm n})\ge \frac{1}{\beta}
\qquad \mbox{for any}~{\bm n}\in 
\cup_{y\in B(x,\delta)\cap \partial D}{\cal N}_y.
\end{align}

\item[{\rm (B')}]
$(\mbox{uniform interior cone condition})$
There exist $\delta>0$ and $0\le\alpha <1$
such that for any $x\in \partial D$
there exists a unit vector $l_x$ such that
$$
C(y,l_x,\alpha)\cap B(x,\delta)\subset \bar{D}
\quad \mbox{for any $y\in B(x,\delta)\cap\partial D$},
$$
where
$C(y,l_x,\alpha)=\{z\in \RR^d \, |\, 
(z-y,l_x)\ge \alpha|z-y|\}$.

\item[{\rm (C)}]
There exists a $C^2_b$ function $f$ on $\RR^d$ 
and a positive constant $\gamma$ such that 
for any $x\in \partial D$, $y\in \bar{D}$, ${\bm n}\in {\cal N}_x$
it holds that
\begin{align}
 \left(y-x,{\bm n}\right)+\frac{1}{\gamma}\left((D
 f)(x),\bm{n}\right)|y-x|^2\ge 0.\label{condition C}
\end{align}
\end{enumerate}
\end{defin}

Note that if $D$ is a convex domain, the condition (A) holds for any
$r_0$ and the condition (C) holds
for $f\equiv 0$.
Also we can prove that
the condition (B') implies condition (B) with the same $\delta$
and $\beta=(1-\alpha^2)^{-1/2}$ by noting that
$\bm{n}_y\in {\cal N}_{y,r}$ is equivalent to
$$
\left(z-y,\bm{n}_y\right)+\frac{1}{2r}|y-z|^2\ge 0
\quad \mbox{for any $z\in \bar{D}$}.
$$

Further, if $D$ is a convex domain in $\RR^2$ or
a bounded convex domain in any dimensions
, then the condition (B) holds.
This is stated in \cite{tanaka}.
Before considering reflecting SDE, let us explain the Skorohod
problem on the multidimensional domain $D$ for which
${\cal N}_x\ne \emptyset$ for all $x\in \partial D$.
Let
$w=w(t)$~$(0\le t\le T)$ be a continuous path on $\RR^d$
with $w(0)\in \bar{D}$.
The pair of paths $(\xi,\phi)$ on $\RR^d$
is a solution of a Skorohod problem
associated with $w$ if the following properties
hold.
\begin{itemize}
 \item[(i)] $\xi=\xi(t)$~$(0\le t\le T)$ is a continuous path 
in $\bar{D}$ with $\xi(0)=w(0)$.
\item[(ii)] It holds that $\xi(t)=w(t)+\phi(t)$
for all $0\le t\le T$.
\item[(iii)] $\phi=\phi(t)$~$(0\le t\le T)$ is a continuous bounded variation 
path on $\RR^d$ such that $\phi(0)=0$ and
\begin{align}
\phi(t)&=\int_0^t{\mathbf n}(s)d\|\phi\|_{[0,s]}\\
\|\phi\|_{[0,t]}&=\int_0^t1_{\partial D}(\xi(s))d\|\phi\|_{[0,s]}.
\end{align}
where ${\mathbf n}(t)\in {\cal N}_{\xi(t)}$ if 
$\xi(t)\in \partial D$.
\end{itemize}
In the above,
the notation $\|\phi\|_{[s,t]}$ stands for the total variation norm
of $\phi(u)$~$(0\le s\le u\le t\le T)$.

Let us consider reflecting SDEs.
Let $\sigma\in C^2_b(\RR^d\to \RR^d\otimes \RR^n)$
and $b\in C^1_b(\RR^d\to \RR^d)$.
Let $\Omega=C([0,\infty)\to \RR^n ; \omega(0)=0)$ and
$P$ be the Wiener measure on $\Omega$.
Let $B(t,\omega)=\omega(t)$~$(\omega\in \Omega)$ be the canonical realization of
Brownian motion.
We consider the reflecting SDE on $\bar{D}$:
\begin{align}
 X(t,x,\omega)&=x+\int_0^t\sigma(X(s,x,\omega))\circ
 dB(s,\omega)+\int_0^tb(X(s,x,\omega))ds+\Phi(t,\omega),\label{reflecting sde}
\end{align}
where $\circ dB(s)$ denotes the Stratonovich integral.
We use the notation $({\rm SDE})_{\sigma,b}$ to indicate this equation.
Note that this usage is different from
that in \cite{aida-sasaki} but I think there are no confusion.
The solution $(X(t),\Phi(t))$ to this equation
is nothing but a solution of the Skorohod problem
associated with
$$
Y(t)=x+\int_0^t\sigma(X(s,x,\omega))\circ
 dB(s,\omega)+\int_0^tb(X(s,x,\omega))ds.
$$
As explained in the Introduction,
if either the condition
\lq\lq (i) $D$ is a convex domain'' or
the condition \lq\lq (ii) 
$D$ satisfies the conditions {\rm (A)} and {\rm (B)}''
holds,
then the strong
solution $X(t)$ to (\ref{reflecting sde}) exists uniquely.
These are due to Tanaka~\cite{tanaka} for (i) and
Saisho~\cite{saisho} for (ii).
See also \cite{lions-sznitman}.
Let $X^N$ be the Wong-Zakai approximation
of $X$.
That is, $X^N$ is the solution to the reflecting differential equation
driven by continuous bounded variation paths:
\begin{align}
 X^N(t,x,\omega)=x+\int_0^t\sigma(X^N(s,x,\omega))
dB^{N}(s,\omega)+\int_0^tb(X^N(s,x,\omega))ds+
\Phi^N(t,\omega),\label{reflecting ode}
\end{align}
where
\begin{align}
 B^{N}(t)&=B(t_{k-1}^N)+\frac{\Delta_N B_k}{\Delta_N}(t-t_{k-1}^N)
\qquad t^N_{k-1}\le t\le t^N_k,\\
\Delta_N B_k&=B(t_{k}^N)-B(t^N_{k-1}),\qquad \Delta_N=\frac{T}{N},\qquad
t_k^N=\frac{kT}{N}.
\end{align}
We may denote $t^N_k$ and $\Delta_N$ by $t_k$ and $\Delta$ respectively.
The solution $X^N$ uniquely exists under conditions (A) and (B) on $D$.
See, {\it e.g.}, \cite{aida-sasaki, saisho}.
Under the convexity assumption of $D$ too,
the solution $X^N$ uniquely exists by the results
in \cite{tanaka}.
In the convex case, we can check the existence in the following different way.
More generally we consider a reflecting differential equation driven by
a continuous bounded variation path
$w_t$:
\begin{align}
 x_t&=x_0+\int_0^t\sigma(x_s)dw_s+\int_0^tb(x_s)ds+\Phi(t)\quad
x_t\in \bar{D}.
\end{align}
The definition of the solution to this equation is similar to
that of the equation previously discussed.
Let $D_R=B(x_0,R)\cap D$.
Then conditions (A) and (B) hold on $D_R$
and the solution, say, $x^R_t$ to reflecting differential equation 
on $D_R$ exists.
Moreover by Lemma~2.4 in \cite{aida-sasaki},
$\|x^R\|_{[0,T]}\le 2(\sqrt{2}+1)
(\|\sigma\|_{\infty}\|w\|_{[0,T]}+\|b\|_{\infty}T)$,
where $\|w\|_{[0,T]}$ denotes the total variation of
$w(t)$~$(0\le t\le T)$ as we already explained
and $\|\sigma\|_{\infty}$ and $\|b\|_{\infty}$
denotes the sup-norm of the operator norm and the Euclidean norm of
$\sigma$ and $b$ respectively.
Thus, we have 
$\max_{0\le t\le T}|x^R(t)-x_0|\le
2(\sqrt{2}+1)
(\|\sigma\|_{\infty}\|w\|_{[0,T]}+\|b\|_{\infty}T)
$
and we can apply the result in the case
where (A) and (B) hold.
Now we are in a position to state our main theorems.

\begin{thm}\label{convex case}
Assume $D$ is convex.
Then,
for any $0<\theta<1$, we have
\begin{align}
\max_{0\le t\le T}
E\left[|X^{N}(t)-X(t)|^2\right]
\le C_{\theta}\cdot \Delta_N^{\theta/2}\label{sup outside t}\\
 E\left[\max_{0\le t\le T}|X^{N}(t)-X(t)|^{2}\right]\le
C_{T,\theta}\Delta_N^{\theta/6}.\label{sup inside t}
\end{align}
\end{thm}

\begin{thm}\label{conditions A B}
Assume the conditions {\rm (A)} and {\rm (B)} hold.
Then for any $\ep>0$, we have
\begin{align}
\lim_{N\to\infty}P\left(\max_{0\le t\le T}|X^N(t)-X(t)|\ge \ep\right)=0.
\end{align}
\end{thm}

\begin{rem}
{\rm Rough path analysis clarifies the meaning of Wong-Zakai approximations.
We refer the readers for basic results of rough path analysis 
to \cite{lyons98, lq, lct, friz-victoir, gubinelli}
and for Wong-Zakai approximations of rough differential equations
driven by fractional Brownian motions
to \cite{friz-o, hu-nualart, deya-neuenkirch-tindel}.
Note that reflecting differential equations driven by rough paths
are defined
and the existence and estimates of the solutions are studied in
the author's recent paper~\cite{aida}.
See also \cite{ferrante-rovira} for reflecting differential equations
driven by fractional Brownian motions whose Hurst parameter
are greater than $1/2$.
}
\end{rem}

\section{Convex domains}

In this section, we prove Theorem~\ref{convex case}.
Below, we use the notation
$$
\|w\|_{\infty,[s,t]}=\max_{s\le u\le v\le t}|w(u)-w(v)|.
$$
The notation $\|w\|_{[s,t]}$ was already defined in 
Section 2.
We can prove the following in the same way as in
the proof of Lemma 2.3 in \cite{aida-sasaki}.

\begin{lem}\label{estimate on local time}
Assume conditions {\rm (A)} and {\rm (B)} hold.
Let $w$ be a $q$-variation continuous path such that
\begin{align}
|w(t)-w(s)|&
\le \omega(s,t)^{1/q}\qquad 0\le s\le t\le T
\end{align}
where $q\ge 1$ and $\omega$ is a control function.
That is, $\omega(s,t)$ is a nonnegative continuous function
of $(s,t)$ with $0\le s\le t\le T$ satisfying
$\omega(s,u)+\omega(u,t)\le \omega(s,t)$ for all 
$0\le s\le u\le t\le T$.
Then the local time $\phi$ of the solution to the
Skorohod problem associated with $w$ has the following estimate.
\begin{align}
\|\phi\|_{[s,t]}
&\le \beta
\left(\left\{\delta^{-1}G(\|w\|_{\infty,[s,t]})
+1\right\}^{q}\omega(s,t)+1
\right)
\left(G(\|w\|_{\infty,[s,t]})+2\right)\|w\|_{\infty,[s,t]},
\label{estimate on local time 1}
\end{align}
where 
\begin{align}
G(a)&=4\left\{1+\beta
\exp\left\{\beta\left(2\delta+a\right)/(2r_0)\right\}
\right\}\exp\left\{\beta\left(2\delta+a\right)/(2r_0)\right\}.
\end{align}
\end{lem}

The above estimate is one of key for the proof in \cite{aida-sasaki}.
Since the unbounded convex domains in $\RR^d$ ($d\ge 3$)
may not satisfy the condition (B), we cannot use this estimate.
However, it is possible to estimate the total variation $\|\phi\|_{[s,t]}$
by $\|w\|_{\infty,[s,t]}$ together with the sup-norm of $\xi$
since we can give an estimate for the numbers $\beta$ and $\delta$ 
in the condition (B) for
bounded convex domains.

\begin{lem}
Let $D$ be a convex domain in $\RR^d$.
Let $x_0\in D$ and assume that there exists 
$R_0>0$ such that $\overline{B(R_0,x_0)}\subset D$.
Let $R\ge R_0$ and define $D_R=D\cap B(R,x_0)$.
The Condition {\rm (B)} holds for the bounded convex domain 
$D_R$ with $\delta=R_0/2$ and $\beta=
\left(1+\left(\frac{2R}{R_0}\right)^2\right)^{1/2}$.
\end{lem}

\begin{proof}
We prove the condition (B').
Let $x\in \partial D_R$.
Let $l_x$ be the unit vector in the direction
from $x$ to $x_0$.
Let $S(x_0)$ be a $d-1$ dimensional ball which is the slice of
the ball $\overline{B(R_0,x_0)}$ by a hyperplane $H(x_0)$
that passes through $x_0$
and is orthogonal to $l_x$.
Let $\alpha=\frac{R}{\sqrt{R^2+(R_0/2)^2}}$.
Then for any point $y\in B(\delta,x)$,
it holds that $C(y,l_x,\alpha)\cap H(x_0)\subset S(x_0)$.
Hence for any $y\in B(\delta,x)\cap \partial D_R$,
$C(y,l_x,\alpha)\cap B(x,\delta)\subset \overline{D_R}$
which implies condition (B').
\end{proof}

\begin{lem}\label{local time estimate convex case}
Let $D$ be a convex domain.
Let $x_0\in D$ and assume that there exists 
$R_0>0$ such that $\overline{B(R_0,x_0)}\subset D$.
Let $w(t)$ $(0\le t\le T)$ be a continuous 
$q$-variation path with the control function $\omega$
on $\RR^d$ with $w(0)\in \bar{D}$ and $q\ge 1$.
Assume that there exists a solution $(\xi,\phi)$ to the
Skorohod problem associated with
$w$.
Then it holds that

\begin{align}
\|\phi\|_{[s,t]}&\le 10\Biggl[
\left\{16R_0^{-1}
\left(
1+4R_0^{-2}
\|\xi-x_0\|_{\infty,[0,T]}^2
\right)^{1/2}+1
\right\}^q\omega(s,t)+1\Biggr]\nonumber\\
&\quad\quad \times\left(1+4R_0^{-2}
\|\xi-x_0\|_{\infty,[0,T]}^2\right)
\|w\|_{\infty,[s,t]}.
\end{align}

\end{lem}

\begin{proof}
Note that $\xi$ is the solution of the
Skorohod problem associated with $w$
on $\overline{D\cap B(x_0,R)}$,
where
$R=\|\xi-x_0\|_{\infty,[0,T]}$.
This domain satisfies (B) with the constants
$\delta$ and $\beta$ specified in the above lemma.
In the lemma, letting $r_0\to\infty$, $G$ reads
\begin{align}
G(a)&=
4\left\{1+\sqrt{1+(2R_0^{-1}R)^2}\right\}.
\end{align}
By applying Lemma~\ref{estimate on local time}, we complete the proof.
\end{proof}

To prove Theorem~\ref{convex case},
we need moment estimates for increments of
$X^N$ and $\Phi^N$.

\begin{lem}\label{XN small interval}
Assume $D$ is a convex domain.
For the Wong-Zakai approximation $X^N$,
we define
\begin{align}
Y^N(t,x,\omega) 
&=x+\int_0^t\sigma(X^N(s,x,\omega))
dB^{N}(s,\omega)+\int_0^tb(X^N(s,x,\omega))ds.
\end{align}

\noindent
$(1)$ For all $p\ge 1$, we have

\begin{align}
E[\|Y^N\|^{2p}_{\infty,[s,t]}] \le C_p|t-s|^p.\label{YN general interval}
\end{align}

\noindent
$(2)$ Let $t_{k-1}\le s<t\le t_k$.
Then we have for all $p\ge 1$,
\begin{align}
& E[|X^N(t)-X^N(s)|^{2p}~|~{\cal F}_{t_{k-1}}]\le C_p|t-s|^{p},
\label{XN small interval 2}\\
& \|\Phi^N\|_{[s,t]}\le C\left(|\Delta B_k|\frac{t-s}{\Delta}+(t-s)\right),
\end{align}
where $C_p$ and $C$ are positive constants.
\end{lem}

\begin{proof}
These assertions can be proved by the same way as the proof
of Lemma 4.3 and Lemma 4.4 in \cite{aida-sasaki}.
We assumed the condition (B) in those lemmas but 
we can argue in the same way since Skorohod equation associated
with the continuous bounded variation path is uniquely solved
under the convexity of $D$.
\end{proof}

\begin{lem}\label{moment estimate}
Assume $D$ is convex.
Let $p\ge 2$ be an integer.
For $0\le s\le t\le T$, we have
\begin{align}
E\left[|X(t)-X(s)|^{p}\right]&\le C_p|t-s|^{p/2},\label{X general interval}\\
E\left[|X^N(t)-X^N(s)|^{p}\right]&\le C_p|t-s|^{p/2},\label{XN general interval}\\
E\left[\|\Phi^N\|_{[s,t]}^{p}\right]&\le C_p|t-s|^{p/2},\label{PhiN general interval}
\end{align}
where $C_p$ is a positive number independent of $N$.
\end{lem}

\begin{proof}
Let $\tau_R=\inf\{t>0~|~X(t,x,w)\notin B(x,R)\}$
and $X^{\tau_R}(t)=X(t\wedge \tau_R)$.
For
(\ref{X general interval}),
it suffices to prove
$E[|X^{\tau_R}(t)-X^{\tau_R}(s)|^{p}|]\le C_p|t-s|^{p/2}$
for all even positive integers $p$ and $0\le s\le t\le T$,
where $C_p$ is independent of $R$.
We prove this by an induction on $p$.
Let $\tilde{b}=b+\frac{1}{2}\tr (D\sigma)(\sigma)$.
By the Ito formula,
\begin{align}
|X^{\tau_R}(t)-X^{\tau_R}(s)|^{2}
& =
2\int_{s\wedge \tau_R}^{t\wedge\tau_R}
(X^{\tau_R}(u)-X^{\tau_R}(s),\sigma(X^{\tau_R}(u))dB(u))
\nonumber\\
&\quad +2\int_{s\wedge \tau_R}^{t\wedge\tau_R}
(X^{\tau_R}(u)-X^{\tau_R}(s),\tilde{b}(X^{\tau_R}(u)))du\nonumber\\
&\quad +
\int_{s\wedge\tau_R}^{t\wedge \tau_R}
\tr\left((\sigma\,{}^t\sigma)(X^{\tau_R}(u))\right)du\nonumber\\
&\quad  +2\int_{s\wedge \tau_R}^{t\wedge \tau_R}
(X^{\tau_R}(u)-X^{\tau_R}(s), d\Phi(u)).
\end{align}
Noting the non-positivity of the term containing
$\Phi$ which follows from the convexity of $D$ and
taking the expectation, we have
\begin{align}
\lefteqn{E\left[|X^{\tau_R}(t)-X^{\tau_R}(s)|^{2}\right]}\nonumber\\
&\le C\int_s^t
E\left[|X^{\tau_R}(u)-X^{\tau_R}(s)|^{2}\right] du
+C(t-s)
\end{align}
which implies $E[|X^{\tau_R}(t)-X^{\tau_R}(s)|^2]
\le C(t-s)$.
Let $p\ge 4$ and suppose the inequality holds for $p-2$.
\begin{align}
\lefteqn{|X^{\tau_R}(t)-X^{\tau_R}(s)|^{p}}\nonumber\\
& =
p\int_{s\wedge \tau_R}^{t\wedge\tau_R}
|X^{\tau_R}(u)-X^{\tau_R}(s)|^{p-2}
(X^{\tau_R}(u)-X^{\tau_R}(s),\sigma(X^{\tau_R}(u))dB(u))
\nonumber\\
&\quad +p\int_{s\wedge \tau_R}^{t\wedge\tau_R}|X^{\tau_R}(u)-X^{\tau_R}(s)|^{p-2}
(X^{\tau_R}(u)-X^{\tau_R}(s),\tilde{b}(X^{\tau_R}(u)))du\nonumber\\
&\quad +
\frac{p}{2}\int_{s\wedge\tau_R}^{t\wedge \tau_R}
|X^{\tau_R}(u)-X^{\tau_R}(s)|^{p-2}
\tr\left((\sigma\,{}^t\sigma)(X^{\tau_R}(u))\right)du\nonumber\\
&\quad +
\frac{1}{2}p\left(p-2\right)\int_{s\wedge \tau_R}^{t\wedge\tau_R}
|X^{\tau_R}(u)-X^{\tau_R}(s)|^{p-4}
|{}^t\sigma(X^{\tau_R}(u)\left(X^{\tau_R}(u)-X^{\tau_R}(s)\right)|^2du\nonumber\\
&\quad  +p\int_{s\wedge \tau_R}^{t\wedge \tau_R}|X^{\tau_R}(u)-X^{\tau_R}(s)|^{p-2}
(X^{\tau_R}(u)-X^{\tau_R}(s), d\Phi(u)).\nonumber
\end{align}
Hence we have
\begin{align}
 E\left[|X^{\tau_R}(t)-X^{\tau_R}(s)|^{p}\right]
&\le
C_p\left(\int_s^t
E\left[|X^{\tau_R}(t)-X^{\tau_R}(s)|^{p-2}\right]+
E\left[|X^{\tau_R}(t)-X^{\tau_R}(s)|^{p-3}\right]
\right)du\nonumber\\
&\le
C_p
\left(\int_s^t
E\left[|X^{\tau_R}(u)-X^{\tau_R}(s)|^{p-2}\right]+
E\left[|X^{\tau_R}(u)-X^{\tau_R}(s)|^{p}\right]
\right)du\nonumber
\end{align}
which implies
\begin{align}
 E\left[|X^{\tau_R}(t)-X^{\tau_R}(s)|^{p}\right]
&\le C_pe^{C_p(t-s)}
\int_s^t
E\left[|X^{\tau_R}(u)-X^{\tau_R}(s)|^{p-2}\right]du\nonumber\\
&\le C_p(t-s)^{p/2}.
\end{align}
This proves (\ref{X general interval}).
Next we prove (\ref{XN general interval}).
Again, is is sufficient to prove the case where $p$ is an even number.
We prove this by an induction on $p$ 
similarly to (\ref{X general interval}).
By Lemma 2.4 in \cite{aida-sasaki},
we have $E[\|X^N\|_{[0,T]}^p]<\infty$
for any $p\ge 1$.
We consider the case where $p=2$.
Let $s=t_l<t_m=t$.
By the chain rule,
\begin{align}
 \lefteqn{|X^N(t)-X^N(s)|^2}\nonumber\\
&
=
2\int_s^t(X^N(u)-X^N(s),\sigma(X^N(u))dB^N(u))
+2\int_s^t(X^N(u)-X^N(s),b(X^N(u)))du\nonumber\\
&\quad +2\int_s^t(X^N(u)-X^N(s), d\Phi^N(u))\nonumber\\
&\le
2\int_s^t(X^N(u)-X^N(s),\sigma(X^N(u))dB^N(u))
+2\int_s^t(X^N(u)-X^N(s),b(X^N(u)))du\nonumber \\
& =:I_1+I_2,
\end{align}
where we have used the non-positivity of the third term
which follows from the convexity of $D$.
We estimate $I_1, I_2$.
We have
\begin{align}
I_1&=
\sum_{k=l+1}^m
2\int_{t_{k-1}}^{t_k}
\left(X^N(u)-X^N(s),
\sigma(X^N(u))
\frac{\Delta B_k}{\Delta}\right)du.
\end{align}
\begin{align}
I_{1,k}&:=\int_{t_{k-1}}^{t_k}
(X^N(u)-X^N(s),\sigma(X^N(u))\frac{\Delta B_k}{\Delta})du\nonumber\\
& =
\left(X^N(t_{k-1})-X^N(s), \sigma(X^N(t_{k-1}))\Delta B_k\right)\nonumber\\
&\quad +
\int_{t_{k-1}}^{t_k}\left(
X^N(u)-X^N(t_{k-1}), \sigma(X^N(t_{k-1}))
\frac{\Delta B_k}{\Delta}\right)du\nonumber\\
&\quad +\int_{t_{k-1}}^{t_k}
\left(X^N(t_{k-1})-X^N(s), \left(\sigma(X^N(u))-
\sigma(X^N(t_{k-1}))\right)\frac{\Delta B_k}{\Delta}\right)\nonumber\\
&\quad +
\int_{t_{k-1}}^{t_k}\left(
X^N(u)-X^N(t_{k-1}), \left(\sigma(X^N(u))-\sigma(X^N(t_{k-1}))\right)
\frac{\Delta B_k}{\Delta}\right)du
\end{align}
By Lemma~\ref{XN small interval} (2),
\begin{align}
E\left[I_{1,k}\right]
&\le C\left(1+E[|X^N(t_{k-1})-X^N(s)|]\right)\Delta\nonumber\\
&\le C\left(\int_{t_{k-1}}^{t_k}\left(E[|X^N(u)-X^N(s)|^2]+1\right)du\right).
\end{align}
Thus, we obtain
\begin{align}
E[|X^N(t)-X^N(s)|^2]&\le
C\left((t-s)+\int_s^tE[|X^N(u)-X^N(s)|^2]du\right).
\label{XN Gronwall}
\end{align}
Again by noting Lemma~\ref{XN small interval} (2),
we see that (\ref{XN Gronwall}) holds for any
$0\le s\le t\le T$.
Applying
the Gronwall inequality, we get
the inequality (\ref{XN general interval}) with $p=2$.
Let $p\ge 4$.
Let $s=t_l<t_m=t$.
By the chain rule,
\begin{align}
 |X^N(t)-X^N(s)|^p&
=
p\int_s^t|X^N(u)-X^N(s)|^{p-2}(X^N(u)-X^N(s),\sigma(X^N(u))dB^N(u))
\nonumber\\
& +p\int_s^t|X^N(u)-X^N(s)|^{p-2}
(X^N(u)-X^N(s),b(X^N(u)))du\nonumber\\
& +p\int_s^t|X^N(u)-X^N(s)|^{p-2}
(X^N(u)-X^N(s), d\Phi^N(u))\nonumber\\
&\le
p\int_s^t|X^N(u)-X^N(s)|^{p-2}
(X^N(u)-X^N(s),\sigma(X^N(u))dB^N(u))\nonumber\\
& +p\int_s^t|X^N(u)-X^N(s)|^{p-2}
(X^N(u)-X^N(s),b(X^N(u)))du\nonumber \\
& =:J_1+J_2,
\end{align}
where we have used the non-positivity of the third term
which follows from the convexity of $D$.
By noting $|X^N(u)-X^N(s)|^{p-1}\le \frac{1}{2}\left(
|X^N(u)-X^N(s)|^p+|X^N(u)-X^N(s)|^{p-2}\right)$ and by
the assumption of induction,
we have
\begin{align}
E[J_2]&\le
C(t-s)^{p/2}+\int_s^t
E[|X^N(u)-X^N(s)|^p]du.
\end{align}
For $J_1$, we have
\begin{align}
 J_1&=
\sum_{k=l+1}^m
p\int_{t_{k-1}}^{t_k}
|X^N(u)-X^N(s)|^{p-2}\left(X^N(u)-X^N(s),
\sigma(X^N(u))
\frac{\Delta B_k}{\Delta}\right)du.
\end{align}
\begin{align}
\lefteqn{\int_{t_{k-1}}^{t_k}
|X^N(u)-X^N(s)|^{p-2}
(X^N(u)-X^N(s),\sigma(X^N(u))\frac{\Delta B_k}{\Delta})du}\nonumber\\
&=
\int_{t_{k-1}}^{t_k}
|X^N(u)-X^N(s)|^{p-2}
\left(X^N(t_{k-1})-X^N(s), \sigma(X^N(t_{k-1}))\frac{\Delta
 B_k}{\Delta}\right)du
\nonumber\\
&\quad +
\int_{t_{k-1}}^{t_k}
|X^N(u)-X^N(s)|^{p-2}\left(
X^N(u)-X^N(t_{k-1}), \sigma(X^N(t_{k-1}))
\frac{\Delta B_k}{\Delta}\right)du\nonumber\\
&\quad +\int_{t_{k-1}}^{t_k}
|X^N(u)-X^N(s)|^{p-2}
\left(X^N(t_{k-1})-X^N(s), \left(\sigma(X^N(u))-
\sigma(X^N(t_{k-1}))\right)\frac{\Delta B_k}{\Delta}\right)\nonumber\\
&\quad +
\int_{t_{k-1}}^{t_k}
|X^N(u)-X^N(s)|^{p-2}\left(
X^N(u)-X^N(t_{k-1}), \left(\sigma(X^N(u))-\sigma(X^N(t_{k-1}))\right)
\frac{\Delta B_k}{\Delta}\right)du\nonumber\\
&=J^k_{1,1}+J^k_{1,2}+J^k_{1,3}+J^k_{1,4}.
\end{align}
We have
\begin{align}
J^k_{1,1}&=J^k_{1,1,1}+J^k_{1,1,2}+J^k_{1,1,3},
\end{align}
where
\begin{align}
 J^k_{1,1,1}&=
\int_{t_{k-1}}^{t_k}
\left\{\int_{t_{k-1}}^u
(p-2)|X^N(r)-X^N(s)|^{p-4}
\left(X^N(r)-X^N(s),\sigma(X^N(r))\frac{\Delta B_k}{\Delta}\right)dr\right\}
\nonumber\\
&\qquad \times
\left(X^N(t_{k-1})-X^N(s), \sigma(X^N(t_{k-1}))\frac{\Delta
 B_k}{\Delta}\right)du,
\end{align}
\begin{align}
 J^k_{1,1,2}&=
\int_{t_{k-1}}^{t_k}
\left\{\int_{t_{k-1}}^u
(p-2)|X^N(r)-X^N(s)|^{p-4}
\left(X^N(r)-X^N(s),b(X^N(r))\right)dr\right\}
\nonumber\\
&\qquad \times
\left(X^N(t_{k-1})-X^N(s), \sigma(X^N(t_{k-1}))\frac{\Delta
 B_k}{\Delta}\right)du,
\end{align}
\begin{align}
J^k_{1,1,3}&=
\int_{t_{k-1}}^{t_k}
\left\{\int_{t_{k-1}}^u
(p-2)|X^N(r)-X^N(s)|^{p-4}
\left(X^N(r)-X^N(s),d\Phi^N(r)\right)\right\}
\nonumber\\
&\qquad \times
\left(X^N(t_{k-1})-X^N(s), \sigma(X^N(t_{k-1}))\frac{\Delta
 B_k}{\Delta}\right)du
\end{align}
By the estimate for $p=2$ and Lemma~\ref{XN small interval} (2),
we have
\begin{align}
E[J^k_{1,1,1}]&\le
C_pE\left[|X^N(t_{k-1})-X^N(s)|^{p-2}\right]\Delta\nonumber\\
& +
\int_{t_{k-1}}^{t_k}\int_{t_{k-1}}^u
E\left[|X^N(r)-X^N(t_{k-1})|^{p-3}|
|X^N(t_{k-1})-X^N(s)|\left(\frac{|\Delta B_k|}{\Delta}\right)^2\right]
drdu\nonumber\\
&\le
C_pE\left[|X^N(t_{k-1})-X^N(s)|^{p-2}\right]\Delta+
C (t_{k-1}-s)^{1/2}\Delta^{(p-1)/2}.
\end{align}
Noting that for any $a>0$,
$\sum_{k=l+1}^m(t_{k-1}-s)^a\Delta\le \int_s^t(u-s)^adu\le
 (t-s)^{a+1}/(a+1)$
and
using the assumption of induction,
\begin{align}
 E[\sum_{k=l+1}^mJ^k_{1,1,1}]&\le
C\sum_{k=l+1}^m
\left\{(t_{k-1}-s)^{(p-2)/2}\Delta+
(t_{k-1}-s)^{1/2}\Delta^{(p-1)/2}\right\}\le C(t-s)^{p/2}.
\end{align}
Similarly,
\begin{align}
E[J^k_{1,1,2}]&\le
C_pE[|X^N(t_{k-1})-X^N(s)|^{p-2}]\Delta^{3/2}
+C(t_{k-1}-s)^{1/2}\Delta^{p/2}.
\end{align}
\begin{align}
& E[J^k_{1,1,3}]\le
C_pE\Bigl
[|X^N(t_{k-1})-X^N(s)|^{p-2}
E\left[\|\Phi^N\|_{[t_{k-1},t_k]}
|\Delta B_k||{\cal F}_{t_{k-1}}\right]\Bigr]\nonumber\\
&\quad +C_p
E\left[
|X^N(t_{k-1})-X^N(s)|
E\left[\max_{t_{k-1}\le r\le t_k}|X^N(r)-X^N(t_{k-1})|^{p-3}
\|\Phi^N\|_{t_{k-1},t_k}|\Delta B_k|\,|\,{\cal F}_{t_{k-1}}
\right]\right]\nonumber\\
&\qquad\le
C_pE[|X^N(t_{k-1})-X^N(s)|]^{p-2}\Delta+
(t_{k-1}-s)^{1/2}\Delta^{(p-1)/2}.
\end{align}
Thus, we have
$E[\sum_{k=l+1}^mJ^k_{1,1,2}]+
E[\sum_{k=l+1}^mJ^k_{1,1,3}]\le
C(t-s)^{p/2}$.

\noindent
We consider the terms $J^k_{1,i}$~$(2\le i\le 4)$.
\begin{align}
E[J^k_{1,2}]&\le\Delta^{-1}
\int_{t_{k-1}}^{t_k}
E\Bigl[|X^N(t_{k-1})-X^N(s)|^{p-2}
E\left[|X^N(u)-X^N(t_{k-1})|^{p-1}|\Delta B_k| 
| {\cal F}_{t_{k-1}}\right]\Bigr]du\nonumber\\
&\quad +\Delta^{p/2}
\nonumber\\
&\le E[|X^N(t_{k-1})-X^N(s)|^{p-2}]\Delta^{p/2}+\Delta^{p/2}.
\end{align}
\begin{align}
E[J^k_{1,3}]&\le
C\Delta^{-1}
\int_{t_{k-1}}^{t_k}
E\Bigl[|X^N(t_{k-1})-X^N(s)|
E\left[|X^N(u)-X^N(t_{k-1})|^{p-1}|\Delta B_k| | {\cal F}_{t_{k-1}}\right]
\Bigr]du\nonumber\\
& \quad +C\Delta^{-1}
\int_{t_{k-1}}^{t_k}
E\Bigl[
|X^N(t_{k-1})-X^N(s)|^{p-1}
E\left[|X^N(u)-X^N(t_{k-1})||\Delta B_k| | {\cal F}_{t_{k-1}}\right]
\Bigr]du\nonumber\\
&\le  CE[|X^N(t_{k-1})-X^N(s)|]\Delta^{p/2}+
CE[|X^N(t_{k-1})-X^N(s)|^{p-1}]\Delta.
\end{align}
\begin{align}
E[J^k_{1,4}]&\le
C\Delta^{-1}
\int_{t_{k-1}}^{t_k}
E\Bigl[|X^N(t_{k-1})-X^N(s)|^{p-2}
E\left[|X^N(u)-X^N(t_{k-1})|^{2}|\Delta B_k| | {\cal F}_{t_{k-1}}\right]
\Bigr]du\nonumber\\
& \quad +C\Delta^{-1}
\int_{t_{k-1}}^{t_k}
E\Bigl[
|X^N(u)-X^N(t_{k-1})|^{p}|\Delta B_k| 
\Bigr]du\nonumber\\
&\le  CE[|X^N(t_{k-1})-X^N(s)|^{p-2}]\Delta^{3/2}+
C\Delta^{(p+1)/2}.
\end{align}
Hence
\begin{align}
E\left[|X^N(t)-X^N(s)|^p\right]&\le
C(t-s)^{p/2}+\int_s^tE\left[|X^N(u)-X^N(s)|^p\right]du.
\label{XN general Gronwall}
\end{align}
By using (\ref{XN small interval 2}),
we see that (\ref{XN general Gronwall}) holds for any
$0\le s\le t\le T$.
By the Gronwall inequality,
we get the desired inequality for $p$
and we complete the proof of (\ref{XN general interval}).
The estimate (\ref{YN general interval}) and the Garsia-Rodemich-Rumsey
estimate imply
the $L^r$-boundedness of the H\"older norm with exponent $1/2-\ep$
of $Y^N$ for any $r\ge 1$ and $0<\ep<1/2$.
Hence,
(\ref{PhiN general interval}) follows from
Lemma~\ref{local time estimate convex case} and (\ref{XN general interval}).
\end{proof}

Thanks to the above estimates, we can 
prove the first main theorem as in \cite{aida-sasaki}.

\begin{proof}[Proof of Theorem~$\ref{convex case}$]
Let $X^N_E(t)$ be the Euler approximation of $X$.
That is, $X^N_E(0)=x$ and $X^N_E$ is the solution to the
Skorohod equation:
\begin{align}
 X^N_E(t)&=X^N_E(t^N_{k-1})+\sigma(X^N_E(t^N_{k-1}))(B(t)-B(t^N_{k-1}))+
\tilde{b}(X^N_E(t^N_{k-1}))(t-t^N_{k-1})\nonumber\\
& \quad+\Phi^N_E(t)-\Phi^N(t^N_{k-1})\qquad t^N_{k-1}\le t\le t^N_k,
\label{euler approximation}
\end{align}
where $\Phi^N_E(t)-\Phi^N(t^N_{k-1})$
is the local time term and
$\tilde{b}=b+\frac{1}{2}\tr (D\sigma)(\sigma)$.
By a similar argument to (\ref{X general interval}) and (\ref{PhiN
 general interval}),
we obtain
\begin{align}
E[\|X^N_E\|_{\infty,[s,t]}^{2p}]&
\le C_p|t-s|^p,\label{estimate for xe}\\
 E\left[\|\Phi^N_E\|_{[s,t]}^{2p}\right]&\le C_p|t-s|^p.\label{estimate for phie}
\end{align}
Hence by the same proof as in \cite{aida-sasaki}, we obtain
there exists $C_{p}>0$ such that
\begin{align}
 E\left[\max_{0\le t\le T}|X^N_E(t)-X(t)|^{2p}\right]
&\le C_{p}\Delta_N^{p}
\end{align}
By these estimates and Lemma~\ref{moment estimate},
we can prove the desired estimates 
as in the same way in \cite{aida-sasaki}.
The proof is simpler than that in \cite{aida-sasaki}
because $f\equiv 0$ when $D$ is convex.
\end{proof}

\section{General domains satisfying conditions (A) and (B)}

In this section, we prove Theorem~\ref{conditions A B}.
The following observation which can be found in Lemma 5.3 in \cite{saisho}
is crucial for our purpose.

\begin{lem}\label{local property A and B}
Assume {\rm (A)} and {\rm (B)} are satisfied on $D$.
Let $\gamma=2r_0\beta^{-1}$.
Then for for each $z_0\in \partial D$ we can find a function
$f\in C^2_b(\RR^d)$ satisfying $(\ref{condition C})$
for any $x\in B(z_0,\delta)\cap \partial D$, 
$y\in \bar{D}$ and ${\bm n}\in {\cal N}_x$.
Moreover the sup-norms $\|D^kf\|_{\infty}$~$(k=0,1,2)$ are bounded by some
constant independent of $z_0$.
\end{lem}

It is stated in Lemma 5.3 in \cite{saisho} that
the conclusion in the above proposition 
holds for $y\in B(z_0,\delta)\cap \bar{D}$.
However, it is obvious to see the same conclusion holds
for any $y\in \bar{D}$.
Thanks to this proposition,
we can localize the problem.
Let us choose a positive number $\delta'<\delta/2$.
For any $z\in \bar{D}$, if $B(z,\delta')\cap \partial D\ne
\emptyset$,
then there exists $z_0\in \partial D$ such that
$\overline{B(z,\delta')}\subset B(z_0,\delta)$.
Next, let $\chi$ be a $C^{\infty}$ function on $\RR^d$ such that
$\chi(x)=1$ for $x$ with $|x|\le \delta'/2$,
$\chi(x)=0$ for $x$ with $|x|\ge 2\delta'/3$.
Let $z\in \bar{D}$ and define
\begin{align}
\sigma^z(x)=\sigma(x)\chi(x-z), \qquad b^z(x)=b(x)\chi(x-z)\qquad x\in \RR^d.
\end{align}
We denote the solution and the
Wong-Zakai approximation to $({\rm SDE})_{\sigma^z,b^z}$ with the starting
point $x$ by $X^{z}(t,x,\omega)$ and $X^{N,z}(t,x,\omega)$
respectively.
By the uniqueness of strong solutions, we have
\begin{itemize}
 \item[(i)] $X^z(t,x,\omega)=X^{N,z}(t,x,\omega)=x$ for all $x\in B(z,\delta'/2)^c$
\item[(ii)] If $x\in B(z,2\delta'/3)$, then both
$X^z(t,x,\omega)$ and $X^{N,z}(t,x,\omega)$ belong to $B(z,2\delta'/3)$
      for all $t$ and $N$.
\end{itemize}

We need a continuous dependence of solutions of reflecting SDE
with respect to the starting point as in the following.
Below, we state it for the particular case ${\rm SDE}_{\sigma^z,b^z}$
but it is easy to extend the result to more general situations.

\begin{lem}\label{localization lemma}
Assume {\rm (A)} and {\rm (B)} hold on $D$.
\begin{enumerate}
\item[{\rm (1)}]For any $p\ge 1$ and $x,y\in \bar{D}$, we have
\begin{align}
 E\left[\max_{0\le t\le T}
|X^z(t,x)-X^z(t,y)|^p\right]&\le C_{p}|x-y|^p.
\label{continuous dependence}
\end{align}
The constant $C_p$ is independent of $z$.
\item[{\rm (2)}] 
Let $0<\theta<1$.
There exists a positive constant $C_{T,\theta}$
such that for any $x,z\in \bar{D}$,
we have
\begin{align}
E\left[\max_{0\le t\le T}
|X^{N,z}(t,x)-X^z(t,x)|^2\right]\le C_{T,\theta}\Delta_N^{\theta/6}.
\label{WZ small ball}
\end{align}
\item[{\rm (3)}] Let $x\in B(z,\delta'/2)$.
Let $\tau(\omega)$ and $\sigma(\omega)$ be the exit time of 
$X(t,x,\omega)$ and $X^z(t,x,\omega)$ respectively
from
$B(z,\delta'/2)$.
Then $\tau(\omega)=\sigma(\omega)$~$P$-a.s. $\omega$ and
$X(t,x,\omega)=X^z(t,x,\omega)$~$(0\le t\le \tau(\omega))$.
\end{enumerate}
\end{lem}

\begin{proof}
(1)
 If $x$ or $y$ belongs to $B(z,2\delta'/3)^c$, 
the assertion is true because of (i) and (ii) above.
Therefore we may assume $x,y\in B(z,2\delta'/3)$.
Suppose $B(z,\delta')\cap \partial D$ is not an empty set.
Then we can pick a point $z_0\in B(z,\delta')\cap \partial D$ such that
$B(z,\delta')\subset B(z_0,\delta)$.
Let $f$ be a function in Lemma~\ref{local property A and B}
associated with $z_0$.
Let 
\begin{align}
 Z^z(t)=X^z(t,x)-X^z(t,y),~~
 \rho^z(t)=e^{-\frac{2}{\gamma}(f(X^z(t,x))+f(X^z(t,y))},~~
k^z(t)=\rho^z(t)|Z^z(t)|^2.
\end{align}
In the calculation below, we omit the superscript $z$ in the
notation $X^z$, and so on.
Let $\tilde{b}=b+\frac{1}{2}\tr (D\sigma)(\sigma)$.
By the Ito formula,
\begin{align}
\lefteqn{d k(t)}\nonumber\\
&=
\rho(t)\Biggl\{
2\Bigl(Z(t),\left(\sigma(X(t,x))
-\sigma(X(t,y))\right)dB(t)\Bigr)\nonumber\\
& \quad +
2\left(Z(t),\tilde{b}(X(t,x))-\tilde{b}(X(t,y))\right)dt +
\|\sigma(X(t,x))-\sigma(X(t,y))\|_{H.S.}^2dt
\Biggr\}\nonumber\\
&\quad +
2\rho(t)\left(Z(t),d\Phi(t,x)-d\Phi(t,y)\right)\nonumber\\
&\quad
-\frac{2\rho(t)}{\gamma}
\left|Z(t)\right|^2\Bigl\{
\left((D f)(X(t,x)),d\Phi(t,x)\right)
+
\left((D f)(X(t,y)),d\Phi(t,y)\right)
\Bigr\}\nonumber\\
&\quad-\frac{2\rho(t)}{\gamma}
\left|Z(t)\right|^2\Bigl\{
\left((D f)(X(t)),\sigma(X(t,x))dB(t)\right)
+
\left((D f)(X(t,y)),\sigma(X(t,y))dB(t)\right)
\Bigr\}\nonumber\\
&\quad +R(t)dt, \label{kt}
\end{align}
where
\begin{align}
R(t)
&=\frac{4\rho(t)}{\gamma}
\Bigl((D f)(X(t,x)),\sigma(X(t,x))\,
{}^t\left(\sigma(X(t,x))-\sigma(X(t,y))\right)
\left(Z(t)\right)\Bigr)dt\nonumber\\
& \quad +\frac{4\rho(t)}{\gamma}
\left((D f)(X(t,y)),\sigma(X(t,y))\,
{}^t\left(\sigma(X(t,x))-\sigma(X(t,y))\right)
\left(Z(t)\right)\right)dt\nonumber\\
& \quad -\frac{2\rho(t)}{\gamma}|Z(t)|^2
\left(\left((D f)(X(t,x)),\tilde{b}(X(t,x))\right)dt+
\left((D f)(X(t,y)),\tilde{b}(X(t,y))\right)dt
\right)\nonumber\\
& \quad -\frac{\rho(t)}{\gamma}|Z(t)|^2
\Bigl\{\tr(D^2f)(X(t,x))\left(\sigma(X(t,x))\cdot,
\sigma(X(t,x))\cdot\right)
\nonumber\\
& \quad \quad\quad \quad\qquad\quad +
\tr(D^2f)(X(t,y))\left(\sigma(X(t,y))\cdot,\sigma(X(t,y))\cdot)\right)
\Bigr\}dt
\nonumber\\
& \quad +\frac{2\rho(t)}{\gamma^2}
\|(D f)(X(t,x))(\sigma(X(t,x)))
+(D f)(X(t,y))(\sigma(X(t,y)))
\|^2
|Z(t)|^2dt.\label{Rt}
\end{align}
Let us take a look at the second and third terms
of (\ref{Rt}).
This term is not equal to 0 when 
$X(t,x)$ or $X(t,y)$ hits $\partial D$.
By the property of $f$,
these terms are negative.
Taking this into account and using the Burkholder-Davis-Gundy
inequality,
we estimate $L^{p}$-norm of $\max_{0\le t\le T'}k(t)$~$(0\le T'\le T)$,
where $p\ge 2$.
Similarly to the proof of Theorem~3.1 in
\cite{aida-sasaki} and Lemma~3.1 in \cite{lions-sznitman},
we have
\begin{align}
 E[\max_{0\le t\le T'}k(t)^p]&\le
C_p|x-y|^{2p}+C_p'\int_0^{T'}E\left[\max_{0\le s\le t}k(s)^p\right]dt
\end{align}
which implies the desired result.

We prove (2).
When $x\notin B(z,2\delta'/3)$, 
$X^z(t,x,\omega)=X^{N,z}(t,x,\omega)=x$ for all $t, N$.
So we assume $x\in B(z,2\delta'/3)$.
If $B(z,\delta')\cap \partial D=\emptyset$, 
by the properties (i) and (ii), 
$X^{N,z}(t,x)$ and $X^z(t,x)$ never hits the
boundary of $D$.
Hence the classical Wong-Zakai theorem implies the assertion.
Suppose $B(z,\delta')\cap \partial D\ne \emptyset$.
Then there exists $z_0\in \partial D$ such that
$\overline{B(z,\delta')}\subset B(z_0,\delta)$.
In \cite{aida-sasaki}, (\ref{WZ small ball}) 
is proved under the conditions (A), (B) and
 (C) on $D$.
By Lemma~3.1, the condition (C) holds locally
in some sense.
Also, $X^{N,z}(t,x), X^z(t,x)\in B(z,2\delta'/3)$.
However, we cannot conclude that the proof 
in \cite{aida-sasaki} works in the present case too.
Because, there, first, we proved that the Euler approximation
converges to the true solution in Theorem 3.1 
and, second,
the difference of the Euler approximation and the Wong-Zakai approximation 
converges to $0$ in Lemma 4.6 in \cite{aida-sasaki}.
In the present case, the Euler approximation solution may exit from
 $B(z,2\delta'/3)$ and reach the boundary of $D$ outside $B(z_0,\delta)$
even if $x\in B(z,2\delta'/3)$.
However, such a probability is small and we can prove (\ref{WZ small
 ball}).
Let us show it more precisely.
Let $X_E^{N,z}(t,x)$ be the Euler approximation
of the solution to $({\rm SDE})_{\sigma^z,b^z}$
with the starting point $x$
associated with the partition
$\{kT/N\}_{k=0}^N$ and $\Phi^{N,z}_E(t,x)$ be the associated local time
 term.
See (\ref{euler approximation}) for the definition of
the Euler approximation.
Let $N$ be a sufficiently large number such that
$\|b\|_{\infty}\Delta_N$ is small.
Then by the estimate (\ref{estimate on local time 1}),
we have
\begin{align}
&P\Bigl(\left\{
\mbox{There exists a time $t\in [0,T]$ such that
$X^{N,z}_E(t,x)\in B(z,\delta')^c$}
\right\}\Bigr)\nonumber\\
&\quad \le 
P\left(\max_{1\le k\le N}\|B\|_{\infty,[(k-1)T/N,kT/N]}\ge
 \ep\delta'\right)\nonumber\\
&\quad \le
P\left(
\|B\|_{{\cal H},\theta}>\ep\delta'\left(\frac{N}{T}\right)^{\theta}
\right)
\le
\exp\left(-C(\ep\delta')^2\left(\frac{N}{T}\right)^{2\theta}\right)
\label{small probability}
\end{align}
where $\ep$ is a small positive number
and $\|~\|_{{\cal H},\theta}$ denotes the H\"older norm with exponent
$\theta$ ($\theta<1/2$).
Thus, combining (\ref{small probability}),
and the moment estimates in Lemma 2.8 and Lemma 3.2 in
 \cite{aida-sasaki}
for $X, \Phi, X^{N,z}_E, \Phi^{N,z}_E$,
by a similar calculation to
the proof of Theorem 3.1,
we obtain
\begin{align}
& E\left[\max_{0\le t\le T'}|X^{N,z}_E(t,x)-X(t,x)|^{2p}\right]\\
&\qquad \le
C_T\Delta_N^p+e^{-C(N/T)^{2\theta}}
+
C_T\int_0^{T'}
E\left[\max_{0\le s\le t}|X^{N,z}_E(s,x)-X(s,x)|^{2p}\right]ds
\label{sigma z euler}
\end{align}
which implies 
$E[\max_{0\le t\le T}|X^{N,z}_E(t,x)-X(t,x)|^{2p}]\le
C_T\Delta_N^p$.
Similarly, the key of the proof
of Lemma 4.6 in \cite{aida-sasaki} is the non-positivity 
of the sum of second and third terms involving local times
$\Phi^N$ and $\Phi^N_E$ in (4.49).
For $({\rm SDE})_{\sigma^z,b^z}$ too,
the corresponding term involving $\Phi^{N,z}$ is non-positive.
For the term $\Phi^{N,z}_E$, by the same reasoning as in 
(\ref{sigma z euler}),
we have
\begin{align}
&E\left[\int_{t_{k-1}}^{t_k}
\left\{\rho^{N,z}(t)\left(Z^{N,z}(t),d\Phi^{N,z}_E(t)\right)
-\frac{\rho^{N,z}(t)}{\gamma}|Z^{N,z}(t)|^2\Bigl((Df)(X^{N,z}_E(t)),
d\Phi^{N,z}_E(t)\Bigr)\right\}\right]\nonumber\\
&\quad \le C_Te^{-C(N/T)^{2\theta}},
\end{align}
where
$
\rho^{N,z}(t)=
\exp\left(-\frac{2}{\gamma}\left(f(X^{N,z}_E(t,x))+f(X^{N,z}(t,x))
\right)\right).
$
Consequently, in a similar way to the proof of Lemma 4.6
in \cite{aida-sasaki},
we obtain for any $0<\theta<1$
\begin{align}
\max_{0\le k\le
 N}E\left[|X^{N,z}(t^N_k)-X^{N,z}_E(t^N_k)|^2\right]
\le C_{\theta}\cdot \Delta_N^{\theta/2}\label{quarter}\\
 E\left[\max_{0\le t\le T}|X^{N,z}(t)-X^z(t)|^{2}\right]\le
C_{T,\theta}\Delta_N^{\theta/6}.
\end{align}
The assertion (3) can be proved by the same argument as
in the proof of Lemma 5.5 in \cite{saisho}.
\end{proof}

\begin{proof}[Proof of Theorem~$\ref{conditions A B}$]
Let $x\in \bar{D}$ and $P_x$ denote the probability law of
the process $X(t,x)$~$(0\le t\le T)$
which exists on $C([0,T]\to \bar{D} ; w(0)=x)$.
Let $c(t)$~$(0\le t\le T)$ be a point of the support of
$P_x$ and 
\begin{align}
 U_{r}(c)=\left\{\omega~\Big |~
\max_{0\le t\le T}|X(t,x,\omega)-c(t)|\le r
\right\}.
\end{align}
It is sufficient to prove that for any $\ep>0$ and $c$
\begin{align}
\lim_{N\to\infty}
P\left(
\left\{\max_{0\le t\le T}|X(t,x)-X^N(t,x)|\ge \ep\right\} 
\cap U_{\delta'/4}(c)
\right)=0.
\end{align}
Let us define a subset of increasing numbers
$\{s_0, \ldots, s_K\}\subset \{t^N_k\}_{k=0}^N$
so that $s_0=0$ and
$s_k=\max\{t^N_l\ge s_{k-1}~|~\max_{s_{k-1}\le t\le t^N_l}
|c(t)-c(s_{k-1})|\le \delta'/8\}$.
For any $c$, if $N$ is sufficiently large, then
the set on the RHS in the definition of $s_k$ is not empty and
$s_K=T$.
Note that the set $\{s_k\}$ and $K$ may depend on $N$ but
$\limsup_{N\to\infty}K<\infty$.
We prove by an induction on
$1\le k\le K$ that for any $\ep>0$ 
\begin{align}
\lim_{N\to\infty}
P\left(
\left\{\max_{0\le t\le s_k}|X(t,x)-X^N(t,x)|\ge \ep\right\}
\cap
U_{\delta'/4}(c)
\right)=0.\label{induction on k}
\end{align}
First, we prove the case $k=1$.
Let 
$s_1^{\ast}=\max\{t~|~\max_{0\le s\le t}|c(s)-x|\le \delta'/8\}$.
Clearly, $s_1\le s_1^{\ast}$ and
$s_1\to s_1^{\ast}$ as $N\to\infty$.
We prove
\begin{align}
\lim_{N\to\infty}
P\left(
\left\{\max_{0\le t\le s_1^{\ast}}|X(t,x)-X^N(t,x)|\ge \ep\right\}
\cap
U_{\delta'/4}(c)
\right)=0.\label{induction 1}
\end{align}
By Lemma~\ref{localization lemma} (3),
we have
\begin{align}
&P\left(
\left\{\max_{0\le t\le s_1^{\ast}}|X(t,x)-X^N(t,x)|\ge \ep\right\}
\cap
U_{\delta'/4}(c)
\right)\nonumber\\
&\quad =
P\left(
\left\{\max_{0\le t\le s_1^{\ast}}|X^x(t,x)-X^N(t,x)|\ge \ep\right\}
\cap
U_{\delta'/4}(c)
\right)\nonumber\\
&\quad =
P\Biggl(
\left\{\max_{0\le t\le s_1^{\ast}}|X^x(t,x)-X^N(t,x)|\ge
 \ep\right\}
\cap
\left\{\max_{0\le t\le s_1^{\ast}}|X^x(t,x)-X^{N,x}(t,x)|\ge
 \delta'/8\right\}\nonumber\\
&\qquad\quad\qquad \cap U_{\delta'/4}(c)
\Biggr)\nonumber\\
&\quad\quad +P\Biggl(
\left\{\max_{0\le t\le s_1^{\ast}}|X^x(t,x)-X^N(t,x)|\ge
 \ep\right\}
\cap
\left\{\max_{0\le t\le s_1^{\ast}}|X^x(t,x)-X^{N,x}(t,x)|\le
 \delta'/8\right\}\nonumber\\
&\qquad\quad\qquad \cap U_{\delta'/4}(c)
\Biggr)\nonumber\\
&\le P\Biggl(
\left\{\max_{0\le t\le s_1^{\ast}}|X^x(t,x)-X^{N,x}(t,x)|\ge
 \delta'/8\right\}
\Biggr)\nonumber\\
&\qquad +P\Biggl(
\left\{\max_{0\le t\le s_1^{\ast}}|X^x(t,x)-X^{N,x}(t,x)|\ge
 \ep\right\}
\Biggr).\label{induction 1 estimate}
\end{align}
Here we have used that
for $\omega$ satisfying
$\max_{0\le t\le s_1^{\ast}}|X^{N,x}(t,x,\omega)-x|\le \delta'/2$,
$X^{N}(t,x,\omega)=X^{N,x}(t,x,\omega)$ holds
for $0\le t\le s_1^{\ast}$.
The estimate (\ref{induction 1 estimate}) and 
Lemma~\ref{localization lemma} (2) implies
the case $k=1$.
We prove (\ref{induction on k}) in the case of $k+1$
assuming the case of $k$.
Let
\begin{align}
 V_{\eta,k}=\left\{\omega~\Big |~
\max_{0\le t\le s_k}|X(t,x)-X^N(t,x)|\le \eta
\right\}.
\end{align}
It suffices to prove
\begin{align}
\limsup_{\eta\to 0}\limsup_{N\to\infty}
P\left(
\left\{\max_{s_k\le t\le s_{k+1}}|X(t,x)-X^N(t,x)|\ge \ep\right\}
\cap
U_{\delta'/4}(c)\cap V_{\eta,k}
\right)=0.
\end{align}
Note that for $t\ge s_k$,
\begin{align}
 X(t,x,\omega)=X(t-s_k,X(s_k,x,\omega),\tau_{k}\omega),\quad
X^N(t,x,\omega)=X^N(t-s_k,X^N(s_k,x,\omega),\tau_{k}\omega),
\end{align}
where $(\tau_k\omega)(t)=\omega(t+s_k)$.
This identity follows from the uniqueness of strong solutions and
$B^N(t,\tau_k\omega)=B^N(s_k+t,\omega)$ for all $k$ and $t\ge 0$.
Hence,
\begin{align}
& P\left(
\left\{\max_{s_k\le t\le s_{k+1}}|X(t,x)-X^N(t,x)|\ge \ep\right\}
\cap
U_{\delta'/4}(c)\cap V_{\eta,k}
\right)\nonumber\\
& \quad \le
P\left(\left\{
\max_{0\le s\le s_{k+1}-s_k}
|X(s,X(s_k),\tau_{k}\omega)-
X(s,X^N(s_k),\tau_{k}\omega)|\ge \ep/2
\right\}\cap U_{\delta'/4}(c)\cap V_{\eta,k}\right)\nonumber\\
&\quad +
P\left(\left\{\max_{0\le s\le s_{k+1}-s_k}
|X(s,X^N(s_k),\tau_k\omega)
-X^N(s,X^N(s_k),\tau_{k}\omega)|
\ge \ep/2\right\}\cap U_{\delta'/4}(c)\cap V_{\eta,k}
\right)\nonumber\\
&\quad :=I_1+I_2,
\end{align}
where we have written 
$X(s_k)=X(s_k,x,\omega)$ and $X^N(s_k)=X^N(s_k,x,\omega)$
for simplicity.
By Lemma~\ref{localization lemma} (1) and the Chebyshev inequality,
we have
$
I_1\le
4\ep^{-2}C_2\eta^2.
$
Let
\begin{align}
& W_{k,\delta',\eta}=
\left\{\max_{0\le s\le s_{k+1}-s_k}
|X(s,X(s_k),\tau_{k}\omega)-
X(s,X^N(s_k),\tau_{k}\omega)|\le \delta'/16
\right\}\cap U_{\delta'/4}(c)\cap V_{\eta,k},\nonumber\\
& I_3=P\left(\left\{\max_{0\le s\le s_{k+1}-s_k}
|X(s,X^N(s_k),\tau_k\omega)
-X^N(s,X^N(s_k),\tau_{k}\omega)|
\ge \ep/2\right\}\cap 
W_{k,\delta',\eta}\right).\nonumber
\end{align}
To prove $\limsup_{\eta\to 0}\limsup_{N\to\infty}I_2=0$,
it suffices to show
$\limsup_{N\to\infty}I_3=0$ for any $\eta$.
We explain the reason.
By Lemma~\ref{localization lemma} (3),
\begin{align}
&P\Biggl(\left\{\max_{0\le s\le s_{k+1}-s_k}
|X(s,X(s_k),\tau_{k}\omega)-
X(s,X^N(s_k),\tau_{k}\omega)|\ge\delta'/16\right\}
\cap U_{\delta'/4}(c)\nonumber\\
&\quad\quad\quad
\cap \left\{|X(s_k)-X^N(s_k)|\le \delta'/8\right\}
\Biggr) \nonumber\\
&\quad
=P\Biggl(\left\{\max_{0\le s\le s_{k+1}-s_k}
|X^{c(s_k)}(s,X(s_k),\tau_{k}\omega)-
X^{c(s_k)}(s,X^N(s_k),\tau_{k}\omega)|\ge\delta'/16\right\}
\cap U_{\delta'/4}(c)\nonumber\\
&\quad\quad\quad
\cap \left\{|X(s_k)-X^N(s_k)|\le \delta'/8\right\}
\Biggr).
\end{align}
By Lemma~\ref{localization lemma} (1), this probability goes to
$0$ as $N\to\infty$
by the assumption of the induction.
Now we estimate $I_3$.
For $\omega\in W_{k,\delta',\eta}$, we have
\begin{align}
 |X(s,X^N(s_k,x,\omega),\tau_k\omega)-c(s+s_k)|\le
\frac{5\delta'}{16} \quad 0\le s\le s_{k+1}-s_k
\end{align}
and so
\begin{align}
 |X(s,X^N(s_k,x,\omega),\tau_k\omega)-c(s_k)|\le
\frac{7\delta'}{16} \quad\quad 0\le s\le s_{k+1}-s_k.\label{induction k}
\end{align}
Here we consider $({\rm SDE})_{\sigma^{c(s_k)},b^{c(s_k)}}$,
where the driving path is $\tau_k\omega$.
Then, for any $\omega\in W_{k,\delta',\eta}$, 
by (\ref{induction k}) and Lemma~\ref{localization lemma} (3),
\begin{align}
X(s,X^N(s_k,x,\omega),\tau_k\omega)=
X^{c(s_k)}(s,X^N(s_k,x,\omega),\tau_k\omega)\quad \quad 0\le s\le s_{k+1}-s_k.
\end{align}
Hence, by a similar argument to the case $k=1$,
we can prove $\limsup_{N\to\infty}I_3=0$
which completes the proof.
\end{proof}

\begin{rem}
 {\rm
As explained in the above proof,
we estimated the difference $X^N-X^N_E$
in Lemma 4.6 in \cite{aida-sasaki}.
However, it is easy to check that
we can estimate the difference 
$X^N-X$ in a similar way to the proof of
$X^N-X^N_E$ and obtain
$\max_{0\le t\le T}E[|X^N(t)-X(t)|^2]\le
 C_{T,\theta}\Delta_N^{\theta/2}$ in the setting
in \cite{aida-sasaki}.
In the proofs of Theorem~\ref{convex case} and Theorem~\ref{conditions A
 B} too,
we can directly estimate the difference
$X^N-X$
in the convex case and 
$X^{N,z}-X^z$ similarly.
By noting this, actually, we do not need to
use the Euler approximation in the above proofs too.
Also, we note that Zhang~\cite{zhang} proved that
the difference $X^N-X$ converges to $0$
without using the Euler approximation
under stronger assumptions than those in \cite{aida-sasaki}.
}
\end{rem}

\end{document}